\documentclass{amsart}
\usepackage{amssymb,amsmath,amscd,xy,graphicx,textcomp}
\newtheorem{theorem}{Theorem}[section]
\newtheorem{lemma}[theorem]{Lemma}
\newtheorem{corollary}[theorem]{Corollary}
\newtheorem{definition}[theorem]{Definition}

\newtheorem{remark}[theorem]{\it Remark}
\newtheorem{example}[theorem]{Example}
\newtheorem{proposition}[theorem]{Proposition}

\xyoption{arrow}

\xyoption{matrix}

\def\C{\mathbb{C}}
\def\R{\mathbb{R}}

\def\ql{\backslash \! \backslash}

\title{Valuations from representation theory and tropical geometry}
\author{Christopher Manon}
\thanks{This work was supported by the NSF fellowship DMS-0902710}

\begin{document}

\begin{abstract}
We recall the space of seminorms discussed by Payne in \cite{P} and define
a slight modification, the space of graded valuations.  After explaining how these spaces relate
to tropical geometry, we describe examples of valuations 
which come from the representation theory of reductive groups.  
\end{abstract}

\maketitle

\tableofcontents

\smallskip

\section{Introduction}

 This note is written with intention of creating a context for recent
results of the author \cite{M1}, \cite{M2}, involving a set of valuations of certain commutative algebras related to the representation theory of a reductive group.  These valuations are intrinsic to the algebra, but were shown to define a point in the tropical variety of the associated ideal of any presentation. We explain this by viewing these objects as generalizations of logarithms of the multiplicative seminorms of Berkovich, also studied recently by Payne, \cite{P}, we will call a logarithm of a seminorm a generalized valuation.  Roughly, the set of generalized valuations $\mathcal{F}_A$ ( $(Spec(A)_{an}$ in \cite{P}) for an algebra $A$ has the property that for any presentation of a subalgebra
  $$
  \begin{CD}
  0 @>>> I @>>> K[X] @>\phi>> A\\
  \end{CD}
  $$
  we obtain a map $\hat{\phi}: \mathcal{F}_A \to tr(I)$ to the tropical variety of the ideal $I,$ this map is continuous with respect to a natural topology.   We view the space $\mathcal{F}_A$ as a step in the direction of creating a theory of tropical geometry which is intrinsic to an algebra or scheme, and is not dependent upon embedding information.   We will also define a slight generalization of the space of generalized valuations, the space of graded valuations $\mathcal{F}_A^s$ for a grading $A = \bigoplus A_s$ of the underlying vector space of the algebra $A.$ This set up appears naturally in several contexts, we will use it to study algebras related to the representation theory of reductive groups, but any context where algebras are multigraded or multifiltered by combinatorial data ought to see an application.  The second part of this note is devoted to showing how to construct valuations from representation theoretic data attached to a ring $A.$  Graded valuations are not always necessary in our examples,  but we will see that it is much easier to prove that a given function is a graded valuation, and that this distinction is almost as good for applications.  We should also mention that rings with other types of extra data have natural graded valuations, see for example \cite{LM} and \cite{KKh} for constructions for rings of global sections of line bundles on projective varieties. 

  \section{Definitions}

In this section we will define the space of generalized valuations and the space
of graded valuations. 

  \begin{definition}
  Let $A$ be an algebra.
  We call $v: A \to \mathbb{R}$ a valuation if the following conditions are satisfied.
  \begin{enumerate}
  \item $v(0) =  -\infty$\\
  \item For any two elements $a, b \in A$ we have $v(a + b) \leq max\{v(a), v(b)\}.$\\
  \item For any two elements $a, b, \in A$ we have $v(ab) = v(a) + v(b)$\\
  \end{enumerate}
  \end{definition}
\noindent
This is the logarithm of the definition of a multiplicative seminorm.
  Using the fact that $\R\cup \{ -\infty\}$ is a tropical algebra with $max = \oplus$ and $+ = \otimes$,  we can reformulate these conditions as follows.
  \begin{enumerate}
  \item $v(0)$ is the additive identity\\
  \item $v(a + b) \oplus v(a) \oplus v(b) = v(a) \oplus v(b)$\\
  \item $v(ab) = v(a) \otimes v(b)$\\
  \end{enumerate}
\noindent
Throughout this note we will move back and forth between the tropical notation and the
classical notation. The set of all such valuations is referred to as $\mathcal{F}_A$ from now on. The requirement that $v(a + b)$ only be less than or equal to $v(a) \oplus v(b)$ may look a little awkward, this is because highest terms can cancel in a sum, we express this with the following proposition.

\begin{proposition}\label{P1}
Let $v$ be a valuation for $A,$ and suppose $v(a+b) \neq v(a) \oplus v(b)$
then $v(a) = v(b).$
\end{proposition}

\begin{proof}
Suppose $v(a) = v(a) \oplus v(b)$ then we have 

\begin{equation}
v(a) = v(a + b - b) \leq v(a+b) \oplus v(b)\\
\end{equation}
\noindent
If $v(a + b) \oplus v(b) = v(a + b),$ then we must have

\begin{equation}
v(a+b) < v(a) \leq v(a + b)\\
\end{equation}

\noindent
which is a contradiction.  Therefore $v(b) = v(a) \oplus v(b) = v(a).$
\end{proof}

\subsection{Cone structures on subsets of $\mathcal{F}_A$}

\begin{definition}
  For $v, w \in \mathcal{F}_A$ we write $v \Rightarrow w$ if $v(a) \leq v(b)$ implies
  $w(a) \leq w(b).$   Tropically this
  is $v(a) \oplus v(b) = v(b)$ implies $w(a) \oplus w(b) = w(b).$
\end{definition}

We say $v$ and $w$ are on the same facet if $v \Rightarrow w$ and $w \Rightarrow v,$ in which case we write $v \Longleftrightarrow w.$  Note that for any positive real number $R \in \mathbb{R},$ $Rv$ is valuation and we have $v \Longleftrightarrow Rv.$  The relation $\Rightarrow$ is transitive, reflexive, and $\Longleftrightarrow$ defines an equivalence relation.

\begin{proposition}
If $v \Rightarrow w_1, w_2$ in $\mathcal{F}_A$ then the function
$w_1 + w_2 : A \to \mathbb{R}_{\geq 0}$ defines a valuation with $v \Rightarrow w_1 + w_2$.
\end{proposition}

\begin{proof}
First we will check the three conditions defining valuations.  The first and third
  conditions follow trivially.  For the second condition, note that we have

  \begin{equation}
  w_1 + w_2(a + b) = w_1(a+b) + w_2(a + b) \leq max\{w_1(a), w_1(b)\} + max\{w_2(a), w_2(b)\}\\ 
  \end{equation}
  
\noindent
  Suppose that $w_1(a) < w_1(b)$ then $v(a) < v(b)$ and we must have $w_2(a) \leq w_2(b),$
  in which case $max\{w_1(a), w_1(b)\} + max\{w_2(a), w_2(b)\} = w_1(b) + w_2(b).$  Similarly
  if $w_1(a) = w_1(b)$ and $w_2(a) < w_2(b)$ we still have $max\{w_1(a), w_1(b)\} + max\{w_2(a), w_2(b)\} = w_1(b) + w_2(b).$ Either way we can write

  \begin{equation}
  w_1 + w_2(a + b) \leq max\{w_1(a), w_1(b)\} + max\{w_2(a), w_2(b)\} \leq \\
\end{equation}
\begin{equation}
max\{w_1(a) + w_2(a), w_1(b) + w_2(b)\} = max\{w_1 + w_2(a), w_1 + w_2(b)\},\\
  \end{equation}
\noindent  
this proves that $w_1 + w_2$ is a valuation. Now if $v(a) = v(b)$ then $w_1(a) = w_1(b)$ and $w_2(a) = w_2(b).$  Similarly if $v(a) < v(b)$ then $w_1(a) \leq w_1(b)$ and $w_2(a) \leq w_2(b).$ This allows us to conclude
  that $v \Rightarrow w_1 + w_2.$
\end{proof}

For a given $v \in \mathcal{F}_A$ the set of all $w$ such that $v \Rightarrow w$ forms a cone in this way. It is of course possible for $v_1 \Rightarrow w$ and $v_2 \Rightarrow w$ for incomparable $v_1$ and $v_2,$ in this way we view $\mathcal{F}_A$ as a complex of cones glued along common boundary cones. We can even give these cones an integral structure by looking at the distinguished valuations $v: A \to \R$ such that $v(a) \in \mathbb{Z}.$

\begin{definition}
For $v, w \in \mathcal{F}_A$ we say $v \rightarrow w$ if for any 
$\{x_1, \ldots, x_n\} = X \subset A$ with associated ideal $I \subset \mathbb{Z}[X]$ in the symbols $X,$ we have $in_v(in_w(I)) = in_v(I).$ 
\end{definition}

\begin{lemma}
If $v \Rightarrow w$ then $v \rightarrow w.$
  \end{lemma}

\begin{proof}
This follows directly from the definitions, in particular for any two elements $a$ and $b$
we must have $v(a) \leq v(b)$ implies $w(a) \leq w(b).$  Hence, if

\begin{equation}
v(\phi(C_i\vec{x}^{\vec{a}(i)})) \leq v(C_j\phi(\vec{x}^{\vec{a}(j)}))\\
\end{equation}
\noindent
we must have

\begin{equation}
w(\phi(C_i\vec{x}^{\vec{a}(i)})) \leq w(\phi(C_j\vec{x}^{\vec{a}(j)}))\\
\end{equation}
\end{proof}

\noindent
In particular, if we have $v \Longleftrightarrow w$ then $Max_v(F) = Max_w(F)$ for all
$F \in I.$  Clearly for any $R \in \R_{\geq 0}$ we have $Rv \leftrightarrow v.$

\begin{proposition}
If $v \rightarrow w_1, w_2$ then $w_1 + w_2$ is a generalized valuation and $v \rightarrow w_1 + w_2.$
\end{proposition}

\begin{proof}
Clearly $w_1 + w_2(0) = - \infty,$ and $w_1 + w_2(ab) = w_1(ab) + w_2(ab) = w_1(a) + w_1(b) + w_2(a) + w_2(b) = w_1 + w_2(a) + w_1 + w_2(b).$  Now consider the elements
$[a + b], [a], [b] \in A$ (brackets for emphasis), which satisfy the relation 

\begin{equation}
[a + b] = [a] + [b]\\
\end{equation}

\noindent 
So by definition we must have containment of maximum terms $Max_v(a+b, a, b) \subset Max_{w_i}(a + b, a, b).$  Hence if $w_1(a) < w_1(b)$ then $w_1(b) = w_1(a + b)$
and $v(b) = v(a + b)$ and $w_2(a) \leq w_2(b).$  In this case we have the following, 

$$w_1 + w_2(a + b) = w_1(a + b) + w_2(a +b) \leq $$
$$max\{w_1(a), w_1(b)\} + max\{w_1(a), w_1(b)\} = $$
$$w_1(b) + w_2(b) = w_1 + w_2(b)$$

\noindent 
A similar statement holds if $w_1(a) > w_1(b)$ and if $w_1(a) = w_1(b).$  This proves
the proposition. 
\end{proof}

\noindent 
The relation $\rightarrow$ is also reflexive and transitive, and $\leftrightarrow$ defines
an equivalence relation. This implies that both $\Rightarrow$ and $\rightarrow$ divide $\mathcal{F}_A$ up into cones, and $\Rightarrow$ refines $\rightarrow.$  Also, fixing a particular set $X \subset A$ defines a relation $\rightarrow_X$ on $\mathfrak{A},$ and each of these is refined by $\Rightarrow,$ but these relations do not capture enough information for the above proof to work (not every set contains $a+b, a$, and $b$), so they do not necessarily form cones.

\subsection{Valuations over a fixed field}

\begin{definition}
Fix a field $K \subset A$ and a valuation $\mathfrak{v}:K \to \R.$
Let $\mathcal{F}_A(K) \subset \mathcal{F}_A$ be the set of generalized
valuations such that $v(k) = -\mathfrak{v}(k)$ for all $k \in K.$ 
\end{definition}

Note that this subset is not always complex of cones, as $v + w$ is never in $\mathcal{F}_A(K)$ if the valuation $\mathfrak{v}:K \to \R$ is nontrivial. The subsets defined by taking a trivial valuation, or the subset of all $v: A \to \R$ such that $v(k) = -r\mathfrak{v}(k)$ for all $k \in K$ for some fixed $ R \in \R_{\geq 0}$ is closed under scaling and addition within the cones, and so inherits the same type of structures, $\Rightarrow, \rightarrow$ found on $\mathcal{F}_A.$  
In \cite{P} it is shown that field extensions $L \subset K$ which preserve the chosen
valuation give the same tropical varieties when the ideal is generated by polynomials
with coefficients in $L.$  Note that this does not imply that $\mathcal{F}_A(L) = \mathcal{F}_A(K),$ as there are valuations of $A$ which agree with the valuation on $L,$ but not $K.$    

\begin{remark}
This extension of scalars result for a field inclusion $L \subset K$ is amusing when $L$ 
has the trivial valuation.  It can be used to prove the theorem in \cite{SpSt} asserting that 
the image of the "tropicalization map" for a field with valuation matches the tropical variety
over a subfield with trivial valuation. 
\end{remark}

\begin{theorem}\label{T1}
For any map from a polynomial ring
$$
\begin{CD}
0 @>>> I @>>> K[X] @>\phi>> A\\
\end{CD}
$$
\noindent
where $X = \{x_1, \ldots, x_n\}$ we get a map
$\hat{\phi}: \mathcal{F}_A(K) \to tr(I)$ defined by

\begin{equation}
v \rightarrow (v(\phi(x_1)), \ldots, v(\phi(x_n))).\\
\end{equation}
\noindent
Furthermore, this map respects the cone structure
on the tropical variety $tr(I).$
\end{theorem}

The last part of the theorem was discussed in the previous section for both the relations $\Rightarrow$ and $\rightarrow.$   Before we launch into a proof of the other part of this theorem let's consider a monomial $x_1^{a_1}\ldots x_n^{a_n} = \vec{x}^{\vec{a}} \in K[X].$  For $v \in \mathcal{F}_A,$ we must have $v(\vec{x}^{\vec{a}}) = \sum v(x_i)a_i,$ by the definition of a valuation. Now we fix a form in the $x_i,$

\begin{equation}
F(x_1, \ldots, x_n) = \sum C_i\vec{x}^{\vec{a}(i)}\\
\end{equation}
\noindent
which is contained in $I.$
First we prove that $\hat{\phi}(v)$ is a tropical
point.

\begin{lemma}
$\hat{\phi}(v)$ weights at least two monomials in
$F(x_1, \ldots, x_n)$ with the highest weight.
\end{lemma}

\begin{proof}
This essentially follows from proposition \ref{P1}, but we will write out the proof.  First note that $v(F(\vec{x})) = -\infty.$ Select the monomial term with the highest weight $C_1\vec{x}^{\vec{a}(1)},$

\begin{equation}
\bigoplus v(C_i\vec{x}^{\vec{a}(i)}) \leq v(C_1\vec{x}^{\vec{a}(1)}).\\ 
\end{equation}

\noindent
Now we can perform the same trick used in the proof of proposition \ref{P1}.

\begin{equation}
v(C_1\vec{x}^{\vec{a}(1)}) = v( C_1\vec{x}^{\vec{a}(1)} + \sum_{i \geq 2} C_i\vec{x}^{\vec{a}(i)} -  \sum_{i \geq 2} C_i\vec{x}^{\vec{a}(i)}) \leq\\
\end{equation}

\begin{equation}
v(F(\vec(x)) \oplus v(\sum_{i \geq 2} C_i\vec{x}^{\vec{a}(i)}) = v(\sum_{i \geq 2} C_i\vec{x}^{\vec{a}(i)}) = \bigoplus v(C_i\vec{x}^{\vec{a}(i)}). \\
\end{equation}
\noindent
This proves that $v(C_1\vec{x}^{\vec{a}(1)}) =  \bigoplus v(C_i\vec{x}^{\vec{a}(i)}),$ 
so at least two terms must be given the highest weight by $v.$ 
\end{proof}

Since $F(x_1, \ldots, x_n)$ was arbitrary, we have  shown the initial ideal of $I$ with respect to $\hat{\phi}(v)$ is monomial free, and therefore this defines a point in $tr(I).$ In the case with trivial valuation, $\hat{\phi}(v)$ is a point in the Gr\"obner fan by construction.  This completes the proof of the theorem.  In \cite{P}, theorem 1.1 Payne uses the maps $\hat{\phi}: \mathcal{F}_A(K) \to tr(I)$ to define
a homeomorphism,

\begin{equation}
\mathcal{F}_A(K) \cong \displaystyle\lim_{\leftarrow} tr(I),\\
\end{equation}

\noindent
where the limit is over all ideals $I$ which present $A$ as a $K$ algebra.  We have seen how the map works in one direction, starting with a point $y_i \in  Lim_{\leftarrow} tr(I)$ we can define a valuation on $A$ by choosing a generating set of $A$ which includes
$f \in A,$ and  letting $v_{y_i}(f)$ be the value of the component of $y_i$ on $f$ on
this generating set, this value is well-defined because of properties of the inverse limit.  To prove that this actually defines a valuation, one looks at generating sets which contain two chosen elements $f$ and $g,$ and their product $fg$ and sum $f + g.$   Indeed, suppose we choose a generating set which includes $f,$ $g,$ and $f + g,$ then the relation 

\begin{equation}
[f+g] = [f] + [g]\\
\end{equation}

\noindent
(brackets added for emphasis) holds in the ideal defining the presentation.  Any 
weighting $w$ in the tropical variety must satisfy $w(f+ g) = w(f) \geq w(g),$
$w(f + g) = w(g) \geq w(f),$ or $w(f) = w(g) \geq w(f + g),$ so we can always write
$w(f + g) \leq max\{w(f), w(g)\}.$  If the generating set of the ideal contains $f,$ $g,$
and $fg$  then any term weighting must satisfy $w(fg) = w(f) + w(g).$ In \cite{P} it is also shown that $\mathcal{F}_A(K)$ surjects onto any given tropical variety associated to a presentation of $A$ by a $K$-polynomial algebra, so any given tropical point comes from a valuation in some sense.    However, this does not imply that valuations are the only way to get tropical points.

\subsection{Generalization to graded valuations}

We define graded valuations.  These objects are useful for capturing 
intrinsic tropical properties of a ring with extra grading structure.  

\begin{definition}
We call $v: A \to \mathbb{R}$ a graded valuation with respect to a grading 
$A = \bigoplus A_s$ if the following are satisfied.
\begin{enumerate}
\item $v(0) = -\infty$\\
\item v(ab) = v(a) + v(b) for $a$ and $b$ homogeneous.\\
\item $v(a + b) \leq v(a) + v(b)$ \\
\end{enumerate}
 
\noindent
We denote the set of all graded valuations with respect to $s$ by $\mathcal{F}_A^s.$
\end{definition}

\noindent
Note that it is easier to prove something is a graded valuation because less
equalities need be checked from condition $2.$ Graded valuations behave much like valuations, in particular they have a similar universality property for tropical varieties of embeddings
given by homogenous parameters. 

\begin{proposition}
Let $v \in \mathcal{F}_A^s(K) \subset \mathcal{F}_A^s$ be a graded
valuation which agrees with $-\mathfrak{v}: K \to \R$  on $K \subset A,$ and let 

$$
\begin{CD}
0 @>>> I @>>> K[X] @>\phi>> A
\end{CD}
$$
\noindent
be a presentation of a subalgebra by $s$-homogeneous elements.
This defines a map $\hat{\phi}: \mathcal{F}_A^s \to tr(I).$
\end{proposition}

\begin{proof}
This follows from the same argument used in theorem \ref{T1}.
\end{proof}

\noindent
As a consequence of this proposition there is a natural map, 

\begin{equation}
\mathcal{F}^s_A(K) \to  \displaystyle\lim_{\leftarrow} [tr(I), X \subset S]\\
\end{equation}

\noindent
where the limit is over all presentations by homogeneous generating sets.   There is also a natural map $\mathcal{F}_A \to \mathcal{F}_A^s$ given by weighting homogeneous terms in $A$ with respect to $s$ with generalized valuations.  In fact, if $s'$ is a grading on $A$
which refines $s$ then we have inclusions,

\begin{equation}
\mathcal{F}_A \subset \mathcal{F}_A^s \subset \mathcal{F}_A^{s'}.\\
\end{equation}
\noindent
In this way every $\mathcal{F}_A^s$ is contained in $\mathcal{F}_A^{\mathcal{B}}$ for 
the grading given by a basis $\mathcal{B} \subset A.$  The relation $\Rightarrow$ makes
sense for $\mathcal{F}_A^s$ as well, and the cone structure it induces is respected by the above inclusions. The $\rightarrow$ relation also makes sense.  

\begin{proposition}\label{conetop}
If $v \in \mathcal{F}_A \subset \mathcal{F}_A^s$ and $v \Rightarrow w$
then $ w \in \mathcal{F}_A.$
\end{proposition}

\begin{proof}
This follows from the definitions. 
\end{proof}

\noindent
The following example shows how one can have a graded valuation which is not a valuation.

\begin{example}
We will look at the polynomial algebra $\C[x, y, z]$  Consider the following assignment
of weights,
$$v(x) = v(y) = v(z) = 1$$
$$v(\C) = 0$$ 
\noindent
We will select the grading given by the one dimensional subspaces defined by each monomial. 
It is obvious that for any two homogeneous elements, the weights of a product add.
Now, suppose we choose to weight $xy + xz$ less than $2.$ Then we get the following. 

\begin{equation}
v(x(y+z)) = v(xy + xz) < 2 = v(x) + v(y + z)\\ 
\end{equation}

\noindent
So this weighting fails as a valuation, but is a graded valuation. 
\end{example}

\noindent

The theory of generalized valuations leaves a lot of room for variations which can take into 
account the variety of structures one can put on a ring.  We have seen the graded variant above, but we could also look at the equivariant case, where we restrict attention to valuations which are preserved by a group action.   In the examples we will see the use of the trivial valuation on $\C.$  We can define the space of (graded) valuations $v$ for a $\C$-algebra by simple stipulating that  $v(C) = 0$ for all non-zero complex numbers $C.$  We also point out that $\R$ can be replaced by other tropical algebras which satisfy similar properties, namely any totally ordered group or semigroup. In the examples we will see the use of $\R_{\geq 0}$ where the role of $-\infty$ is filled by the lower bound $0.$  Indeed, any well-ordered tropical algebra will do, with the lower bound as the additive identity.

  \section{The functor of valuations}

  We now turn our attention to the properties of $\mathcal{F}_A$ and $\mathcal{F}_A^s$ as the algebra $A$ is allowed to vary.  If $f:A \to B$ is a ring homomorphism the pullback  of a valuation may not continue  to be a valuation.  This is because the map $f$ may have a nontrivial kernel, which must map to the lower bound of the tropical algebra being used.  Take this algebra to be $\R_{\geq 0},$ with lower bound $0.$ If $a$ is any element then $v(a) = v(ai) - v(i)$ for any element $i$ in the kernel of the map, this is $0.$
For this reason we focus on injections of algebras, where kernels cannot cause problems. 

  \begin{proposition}
  If $f: A \to B$ is an injection of algebras, then for any valuation
  $v \in \mathcal{F}_B,$ the pullback $f^*(v)$ is a valuation.
 Furthermore, if $v \Rightarrow w$  then $f^*(v) \Rightarrow f^*(w),$ 
and if $v \rightarrow w$  then $f^*(v) \rightarrow f^*(w).$
  \end{proposition}

  \begin{proof}
 We need only note that by definition, any product
  of elements $a\times b$ in $A$ must have a highest weight term with
  weight equal to $v(a) + v(b) = f^*(v)(a) + f^*(v)(b).$  Also by definition
  if $v(a) \leq v(b)$ then $w(a) \leq w(b),$ so if $f^*(v)(a) \leq f^*(v)(b)$ then
  $f^*(w)(a) \leq f^*(w)(b).$ 
  \end{proof}

\noindent
 It is also easy to check that $(f\circ g)^* = g^* \circ f^*$ for injections $f$ and $g.$
These facts have a graded form as well. 

\begin{proposition}
Let $f: A \to B$ be an injection of algebras which respects underlying gradings on 
$A$ and $B.$

\begin{enumerate}
\item $A = \bigoplus A_s,$ $B = \bigoplus_t B_t$\\
\item If $a \in A$ is homogeneous then  $f(b) \in B$ is homogeneous.\\
\end{enumerate}

Then for a graded valuation $v$ on $B,$ the function $f^*(v)$ is a graded
valuation. 
\end{proposition}

\noindent
The inclusion $\mathcal{F}_A^s \subset \mathcal{F}_A^{s'}$ for a grading 
$s'$ which refines $s$ is a special case of this proposition.

\section{Examples}
In this section we will go through two examples and an amusing generalization of the objects we have been discussing.

\begin{example}
Consider the injection of rings defined by inverting an element.

$$
\begin{CD}
A @>i_f>> \frac{1}{f}A\\
\end{CD}
$$
\noindent
Note that we must have $0 = v(1) = v(f \frac{1}{f}) = v(f) + v(\frac{1}{f}).$
 Suppose $v$ and $w$ are two distinct valuations on $\frac{1}{f}A,$ then for some element
we must have

\begin{equation}
v(a_0 + \frac{1}{f}a_1 + \ldots + \frac{1}{f^n}a_n) \neq w(a_0 + \frac{1}{f}a_1 + \ldots + \frac{1}{f^n}a_n)\\
\end{equation}
\noindent

If $v(f) \neq w(f)$ then these elements define distinct elements in $\mathcal{F}_A,$
suppose then that $v(f) = w(f).$  We also must have

\begin{equation}
v(a_0 + \frac{1}{f}a_1 + \ldots + \frac{1}{f^n}a_n) + nv(f) =  v(f^n(a_0 + \frac{1}{f}a_1 + \ldots + \frac{1}{f^n}a_n)) = v(f^n a_0 + f^{n-1}a_1 + \ldots + a_n).\\
\end{equation}
\noindent
Since $v(f) = w(f),$ this implies that $v$ and $w$ must still differ on an element of $A.$   This implies that

\begin{equation}
i_f^*: \mathcal{F}_{\frac{1}{f}A} \to \mathcal{F}_A\\
\end{equation}
\noindent
is an injection of complexes of cones.  This generalizes to an inversion of
any set of elements.
\end{example}

\begin{example}
We'll consider the polynomial algebra $\mathbb{C}[t].$  By definition, the
value $v(t)$ determines $v(t^n) = n v(t).$   For a general polynomial
$p(t) = a_nt^n + \ldots + a_0$ we have the following,

\begin{equation}
v(t^n) = v(a_nt^n) = v(p(t) - (p(t) - a_nt^n)) \leq v(p(t)) \oplus v(p(t)  - a_nt^n)\\
\end{equation}

\begin{equation}
v(p(t)) \leq v(a_nt^n) \oplus v(p(t) - a_nt^n)\\
\end{equation}

\begin{equation}
v(p(t) -a_nt^n) < v(a_nt^n)\\
\end{equation}
\noindent
this implies that

\begin{equation}
v(p(t)) \leq v(t^n) \leq v(p(t))\\
\end{equation}
\noindent
so $v(p(t)) = v(t)deg(p(t)).$  This proves that $\mathcal{F}_{\mathbb{C}[t]} = \mathbb{R}_{\geq 0}.$\\   We could have also proved this using the fact that any polynomial over $\mathbb{C}$ factors, but the content of this example does not depend on the field being algebraically closed.
\end{example}

\begin{example}
We have already mentioned that the tropical algebra $\R \cup \{-\infty\}$ is not
necessary to the formulation of generalized valuations.  Suppose for example we
took a noncommutative tropical algebra, a general totally ordered semigroup $S,$
with a lower bound $O$, and define tropical operations $s_1 \otimes s_2 = s_1s_2$ and $s_1 \oplus s_2 = sup\{s_1, s_2\}.$  We could look at functions $v: A \to S$ such that the following are satisfied. 
\begin{enumerate}
\item $v(0) = O$\\
\item $v(ab) = v(a) \otimes v(b)$\\
\item $v(a + b) \leq v(a) \oplus v(b)$\\
\end{enumerate}

\noindent
We call this the set of generalized valuations of $A$ with coefficients in $S,$
$\mathcal{F}_{A, S}.$
We could then look at the ideal defined by a subset $X \subset A$ by a map 
from the noncommutative polynomial algebra on $X.$

$$
\begin{CD}
0 @>>> I @>>> \C<X> @> \phi >> A\\
\end{CD}
$$

\noindent
Each such $v$ defines a partial ordering on the free noncommutative monoid on $X$ symbols.

$$
\begin{CD}
<X> @>>> \C<X> @>>> A @>v>> S\\
\end{CD}
$$
\noindent 
Much of what we've formulated above carries over formally to this case.
From the conditions above one can prove that for each form $F(X) \in I$
there are at least two monomials in $F$ which are given the maximum value
by $v.$  One could take the set of all monoidal maps $Hom(<X>, S)$ and the
subset $tr_S(I) \subset Hom(<X>, S)$ of maps $v$ such that $in_v(I)$
is monomial free, there is a map 

\begin{equation}
\hat{\phi}: \mathcal{F}_{A, S} \to tr_S(I) \subset Hom(<X>, S)\\
\end{equation}

\end{example}

\section{Examples from representation theory}
In this section we will discuss how to obtain graded valuations from 
some familiar rings related to the representation theory of reductive groups.
The examples we will discuss will all involve an algebra $A,$ with a grading
of the underlying vector space given by a monoid $\mathcal{S},$

$$A = \bigoplus_{s \in \mathcal{S}} A_s.$$

\noindent
Orderings $v:\mathcal{S} \to \R$ will be given such that multiplication in $A$
is lower-triangular with respect to $v.$  So the weight of a sum of two elements from
different graded pieces can then be defined as the max, $v(a + b) = v(a) \oplus v(b).$
The following is a consequence of these conventions. 

\begin{proposition}
If $A,$ $v$ and $\mathcal{S}$ are as above then $v$ defines a graded valuation on 
$A.$  
\end{proposition}

\noindent
This situation is made more interesting if the valuation and the monoid behave well together.  For instance, if for any two elements $a \in A_{s_1}$ and $b \in A_{s_2}$ we have $ab \in A_{s_1 + s_2} \oplus \ldots,$ with
$(ab)_{s_1 + s_2} \neq 0$ and $v(s_1 + s_2) > v(s)$ for the other $s$ appearing
in the decomposition of $ab.$ If these conditions are satisfied for $A$ then it is also true 
for graded subalgebras of $A.$  If the grading $S$ has only $1$-dimensional components,
then graded valuations that possess this property define filtrations on $A$ which 
have monoidal associated graded algebras, as do subalgebras which respect the grading. 
With the next theorem we see that with some assumptions we can gurantee the existence
of generalized valuations (that need not be graded),  this is satisfied by special variants of the valuations constructed in our examples and the valuations from algebraic geometry presented in \cite{LM} and \cite{KKh}. 

\begin{theorem}\label{monoid}
Suppose $A$ has a direct sum decomposition $\bigoplus_S A_s$ by the elements
of a monoid $S,$ and the following are satisfied. 

\begin{enumerate}
\item the expansion of $a_{s_1}b_{s_2}$ contains a component of weight $s_1+ s_2$\\
\item there is a total order weighting $w$ of $S$\\ 
\end{enumerate} 
\noindent
Then any $v:A \to \R$ for which $w \Rightarrow v$ defines a generalized valuation of $A.$
\end{theorem}

\begin{proof}
It suffices by proposition \ref{conetop} to prove this for $w.$
Suppose that for inhomogenous elements $A = a_{s_1} + \ldots + a_{s_n}$ and $B = b_{t_1} + \ldots + b_{t_m}$ that $w(AB) < w(A) + w(B)$ (note that this is all that could go wrong).
For each pair let $a_{s_i}b_{t_j} = \sum c^k_{s_i, t_j}$ be the expansion into
homogenous elements, with $c_{s_i, t_j}^{top}$ the highest component by $w,$
which must have a grading of $s_i + s_j.$
Let $s_1$ and $t_1$ have the highest grades with respect to $w.$   
We must have that the sum of all elements of grade $s_1 + t_1$ is $0.$  We can assume without loss of generality that all of these components are of the form $c_{s_i, t_j}^{top}.$  To see this note that if the grade of $c_{s_i, t_j}^k$ is $s_1 + t_1$ then we must have $w(s_1 + t_1) < w(s_i + t_j)$ unless $k =$ $top.$  So we have

\begin{equation}
\sum c_{s_i, t_j}^{top} = 0\\
\end{equation}

\noindent
for some collection $\{(i, j)\} \subset [1, n]\times [1, m],$ with $s_1 + t_1 = s_i + t_j,$ 
and these are all the components of weight $s_1 + t_1.$    
This implies that $w(s_1) = w(s_i)$ and $w(t_1) = w(t_j)$ for all $i, j,$ because if any of these
were inequalities, $s_1 + t_1$ would not be highest by $w,$ or $s_1 + t_1 \neq s_i + t_j.$ 
But $w$ is a total ordering, so we must have $s_1 = s_i$ and $t_1 = t_j,$ and we must have
$v(A'B') \neq v(A') + v(B')$ for pure elements $A'$ and $B',$ a contradiction.  So, the highest
component of $AB$ by $w$ (therefore with unique $S$ weight) cannot vanish.
\end{proof}

\noindent 
From this theorem we can conclude that weakenings of the total orderings of the dual canonical basis and gradings of branching algebras to be discussed below are generalized valuations.  For dual canonical bases, the total orderings themselves give toric degenerations 
of the associated algebra.  We also mention the following corollary of the proof of theorem \ref{monoid}. 

\begin{corollary}\label{monbran}
If the associated graded algebra of an algebra $A = \bigoplus_S A_s$ with respect to $w: S \to \R$ is a domain, and $w$ well-orders the elements in the expansion of any product, then $w$ is a generalized valuation. 
\end{corollary}

\begin{proof}
Suppose $w$ is not a generalized valuation, then some product of inhomogeneous elements $AB = (\sum a_{s_i})(\sum b_{t_j}) = \sum a_{s_i}b_{t_j} = \sum c_{s_i, t_j}^k$ must have all of its top weighted terms cancel out.  By the same reasoning as in the proof of theorem \ref{monoid} we have that the product of the sums of the subsets of the top $w$-weighted homogeneous terms of $A$ and $B,$ $A' = \sum_{top} a_{s_i}, $ $B' = \sum b_{t_j}$ must 
have $v(A'B') < v(A) + v(B).$ Therefore this product vanishes in the associated graded algebra. (it may be instructive to see why the product $AB$ does not necessarily vanish in the associated graded)
\end{proof}

\noindent
This corollary can be used to show that special graded valuations on branching algebras
are actually generalized valuations, we will mention more below.

\subsection{Rings with a reductive group action}
In this example we will follow the books by Grosshans and Dolgachev, \cite{Gr} and \cite{D}.  Let $A$ be any commutative algebra over $\C$ with an action of a reductive group $G.$  
Then $A$ has a direct sum decomposition into weighted components from the category of finite
dimensional representations of $G$.

\begin{equation}
A = \bigoplus_{\lambda \in \mathcal{S}_G} A^{U_+}(\lambda) \otimes V(\lambda^*)\\
\end{equation}

\noindent
Here, $A^{U_+}(\lambda)$ is the $\lambda$ weight component of the invariants of $A$
with respect to the action the maximal unipotent $U_+ \subset G.$  Let $\mathfrak{h}$ be the Cartan subalgebra corresponding to a chosen maximal torus of $G.$  Multiplication in this algebra is lower triangular in the sense that if $a$ and $b$ are from the $\lambda$ and $\gamma$ components of $A$ then $ab$ decomposes into homogeneous elements from components of weights less than or equal to $\lambda + \gamma$ as dominant weights. 
Any functional $\mathfrak{h}^* \to \R_{\geq 0}$ which respects the ordering on dominant weights then defines graded valuation of $A$ for the $G$-decomposition.  

More generally, one could also start with a graded valuation on $A^{U_+}$ with respect
to the grading $A^{U_+} = \bigoplus_{\lambda \in \mathcal{S}_G} A^{U_+}(\lambda)$
and define a graded valuation on $A$ by ignoring the $V(\lambda^*$ components of
elements in the $A^{U_+}(\lambda) \otimes V(\lambda^*)$ summand.  Note that graded
valuations constructed this way are always $G$-invariant.

\subsection{Dual canonical bases}
In this example we use the same set up as the last example, with $A$ a $\C$-algebra
which carries the action of a reductive group $G,$ we follow the paper of Alexeev and Brion, \cite{AB}.  Each representation $V(\lambda)$ of $G$ has a distinguished basis $B(\lambda)$, called the dual canonical basis.  This basis was defined by Lusztig and has been studied by many authors, see for example \cite{AB} and \cite{BZ}.  We can refine the grading of the previous example to the one given by the components $A^{U_+}(\lambda)$ tensored with members of $B(\lambda).$  

\begin{equation}
A = \bigoplus_{\lambda, \phi \in \mathcal{S}_G, B} A^{U_+}(\lambda) \otimes \C b_{\lambda, \phi}\\
\end{equation}

\noindent
Multiplication for the dual canonical basis is lower triangular with respect an ordering of weights called string parameters which index the basis, these depend on a chosen decomposition of the longest element of the Weyl group, see \cite{AB} or \cite{BZ2}.  Any functional $h: B \to \R_{\geq 0}$ which respects this ordering then gives a graded valuation.  Note that this example specializes to the previous example for any functional which only sees the dominant weight information.


\subsection{Branching algebras}
Branching algebras are discussed in both \cite{M1} and \cite{M2}, they are multigraded
commutative algebras over $\C,$ where the dimensions of the graded components give the
branching multiplicities for a morphism of reductive groups.  For a reductive group $G,$ let
$C_G$ denote the cone of dominant $G$-weights.  Let $U_+ \subset G$ be a maximal unipotent subgroup, we let $R(G)$ denote the commutative algebra defined as follows. 

\begin{equation}
R(G) = \C[G / U_+] = \C[G]^{U_+} = [\bigoplus_{\lambda \in C_G} V(\lambda^*)\otimes V(\lambda)]^{U_+} = \bigoplus_{\lambda \in C_G} V(\lambda^*) \\
\end{equation} 

\noindent
For a map of reductive groups $\phi: H \to G$ we define the full branching algebra for
$\phi$ as follows. 

\begin{equation}
\mathfrak{A}(\phi) = [R(H)\otimes R(G)]^H = R(G)^{U(H)_+}\\
\end{equation}

\noindent
As vector spaces, these algebras decompose into meaningful graded components.

\begin{equation}
\mathfrak{A}(\phi) = \bigoplus_{\mu, \lambda \in C_H\times C_G} Hom_H(W(\mu), V(\lambda))
\end{equation}

\noindent
We have called these full branching algebras because multigraded subalgebras of $\mathfrak{A}(\phi)$ are also meaningful for representation theory, see \cite{M2}.  The reader's
favorite subcones of $C_H\times C_G$ define subalgebras of $\mathfrak{A}(\phi)$ which
share many useful features.  In particular the subcone $\{0\}\times C_G$ defines the subalgebra of $H$-invariants in $R(G)$ with respect the action through $\phi.$
Now take $\phi = \pi \circ \psi$ to be any factorization of $\phi$ in the category of
reductive groups.
$
\begin{CD}
H @>\psi>> K @>\pi>> G\\
\end{CD}
$
\noindent
On the level of vector spaces this refines the direct sum decomposition with respect
to intermediate branching over $K.$

\begin{equation}
Hom_H(W(\mu), V(\lambda)) = \bigoplus_{\eta \in C_K} Hom_H(W(\mu), Y(\eta)) \otimes Hom(Y(\eta), V(\lambda))\\
\end{equation}

\noindent 
Multiplication with respect to this refinement is "lower triangular" in the sense that the multiplication of two elements, 
$$
\begin{CD}
W(\mu_1) @>f_1>> Y(\eta_1) @>g_1>> V(\lambda_1)\\
\end{CD}
$$
$$
\begin{CD}
W(\mu_2) @>f_2>> Y(\eta_2) @>g_2>> V(\lambda_2)\\
\end{CD}
$$
\noindent
Yields a sum of homogeneous terms of the form 
$$
\begin{CD}
W(\mu_1 + \mu_2) @>f>> Y(\eta) @>g>> V(\lambda_1 + \lambda_2)\\
\end{CD}
$$

\noindent
with $\eta \leq \eta_1 + \eta_2$ as dominant weights.  Also, there is always a nonzero
highest weight term with $\eta = \eta_1 + \eta_2,$ of the form 
$$
\begin{CD}
W(\mu_1+\mu_2) @>f_1\circ f_2>> Y(\eta_1 + \eta_2) @>g_1 \circ g_2>> V(\lambda_1 + \lambda_2)\\
\end{CD}
$$
\noindent
where $f_1 \circ f_2$ and $g_1 \circ g_2$ are the multiplications in the branching algebras 
$\mathfrak{A}(\psi)$ and $\mathfrak{A}(\pi)$ respectively.  One can now choose functionals
$\vec{h}$ on the monoid $C_H\times C_K\times C_G$ to produce graded filtrations of $\mathfrak{A}(\phi).$  It is shown in \cite{M1} and \cite{M2} that if the functional $\vec{h}$ is non-negative on all positive roots, it defines a graded valuation. Furthermore, if the functional 
is strictly positive on positive roots then it defines a filtration with a remarkable associated graded.

\begin{equation}
 gr_{\vec{h}}[\mathfrak{A}(\psi\circ \pi)] = [\mathfrak{A}(\psi) \otimes \mathfrak{A}(\pi)]^T
\end{equation}

\noindent
Here $T$ is a maximal torus of $K,$ which picks out the correct tensors of graded components.
These are the generalized valuations mentioned previously in the discussion of corollary \ref{monbran}.  The same reasoning can be used to show that graded valuations defined similary for the algebra of conformal blocks and studied in \cite{M1} are actually generalized valuations. For each factorization of a morphism $\phi = \pi_1 \circ \ldots \circ \pi_k$ in the category of reductive groups we obtain a cone of suitable functionals.   Given a functional $\vec{h} = (h_1, \ldots h_{k+1})$ for a factorization $(\pi_1, \ldots, \pi_k)$  we can set a component $h_i$ to $0$ and obtain a new functional $\hat{h}_i$ for the factorization 
$(\pi_1, \ldots, \pi_{i+1} \circ \pi_i, \ldots, \pi_k).$  
In this way some of the structure of category of reductive groups is realized in a polyhedral complex of graded valuations on branching algebras.   Also, multigraded subalgebras of the full branching algebra inherit these valuations by the functor properties, in particular this holds for the algebra of invariants $R(G)^H.$  

An amusing special case of a branching algebra filtration as above is given by observing that any representation of a reductive group $V$ defines a map $G \to GL_n$ where
$n = dim(V).$  This map factors the map from the trivial group $1 \to GL_n,$ and so 
defines a multifiltration of $\mathfrak{A}(1 \to GL_n) = R(GL_n).$ In this way any representation of a reductive group can be associated with a cone of graded valuations
of $R(GL_n)$ for some $n.$

\subsection{More valuations from branching structures}
Let $A$ be an algebra over $\C$ with an action of a reductive group $G,$ and let $\phi: H \to G$ be a morphism of reductive groups.  We consider the algebra $[R(H)\otimes A]^H.$ 
From the direct sum decomposition of $A$ we get the following direct sum decomposition of $[R(H)\otimes A]^H.$

\begin{equation}
[R(H)\otimes A]^H = \bigoplus_{\mu, \lambda \in C_H \times C_G} A^{U_+}(\lambda^*) \otimes Hom_H(W(\mu), V(\lambda))\\
\end{equation}

\noindent
One can prove that multiplication $[R(H)\otimes A]^H$ has a similar "lower triangular" property
with respect to the dominant weight data in this decomposition.  
In this way $[R(H)\otimes A]^H$ inherits the valuations on $\mathfrak{A}(\phi)$ discussed
in the previous example.   The same holds for multigraded graded subalgebras of $[R(H) \otimes A]^H,$ such as the algebra of invariants, $A^H.$

\begin{equation}
A^H = \bigoplus A^{U_+}(\lambda^*)\otimes Hom_H(\C, V(\lambda))\\
\end{equation}

\noindent
Incidentally, both of these algebras also inherit any graded valuation of $A^{U_+}$
as in the first example above. 

\subsection{More valuations from dual canonical bases}
The functor properties of graded valuations are useful for passing combinatorial structures
from a ring to a subring.  We return to the dual canonical bases, discussed in a previous example.  Recall that these define a direct sum decomposition of an irreducible representation $V(\lambda) = \bigoplus \C b_{\phi},$ and therefore define a grading on the underlying vector space of $R(G)$ by the monoid of string parameters for a chosen decomposition of the longest element of the Weyl group.  Recall also that for suitably assigned orders on these parameters, the multiplication in $R(G)$ is "lower triangular."  Both of these properties therefore pass to subrings which respect this multigrading.  Let $G' \subset G$ be a levi subgroup of $G,$ which corresponds to a subset of the positive roots $\mathcal{R}_{G'} \subset \mathcal{R}_G.$  The subspaces

\begin{equation}
Hom_H(W(\mu), V(\lambda)) = \{v | e_i(v) = 0, h_i(v) = \mu(h_i)v, \alpha_i \in R_{G'}\} \subset V(\lambda),\\
\end{equation}

\noindent
inherit a basis from the dual canonical basis of $V(\lambda),$ where $e_i$ is the raising operator and $h_i$ is the generator of the Cartan subalgebra of $\mathfrak{g}$ corresponding
to the root $\alpha_i.$  Therefore, for a Levi subgroup $G' \subset G,$ we can construct
graded valuations acting on the string parameters of the dual canonical basis, and the associated graded algebras of these multifiltrations are monoidal, as in \cite{AB}.  For a flag of Levi subgroups $G_0 \subset G_1 \subset, \ldots, \subset G,$ this can be used to construct
graded valuations which refine those constructed in the previous example on the branching algebra.  

\bigskip

By a similar argument, the same is true for other rings that inherit the dual canonical basis. Consider the algebra $R(G)\otimes \C[T],$ for a maximal torus $T \subset G.$  We will let $T$ act on $R(G)$ on the left and on $T$ by $t^{-1},$ and consider the subring of invariants.

\begin{equation}
[R(G)\otimes \C[T]]^T = \bigoplus_{\mu, \lambda} V_{\mu}(\lambda) \otimes \C t^{-\mu}\\
\end{equation}

\noindent
Here $V_{\mu}(\lambda)$ is the subspace of $V(\lambda)$ with left $T-$character equal 
to $\mu.$  It is known that this subspace inherits the dual canonical basis.  This ring is particularly interesting because it contains the projective coordinate rings of weight varieties
of $G /B$ as multigraded subalgebras.  By the functor property, these rings inherit 
graded valuations from $R(G) \otimes \C[T]]^T.$

\bigskip

Now we consider opposite maximal unipotents $U_{+}$ and $U_{-}$ in $G,$
and the map of varieties, 

$$
\begin{CD}
T\times G/U_+ @>>> G/U_+ \times G/U_+ @>>> U_{-} \ql [G/U_+ \times G/U_+]\\
@AAA @AAA @AAA \\
T \times U_{-}\times T @>>> U_{-} \times T \times U_{-} \times T @>>> U_{-} \ql [U_{-} \times T \times U_{-} \times T]\\
\end{CD}
$$

\noindent
where a calculation on Lie algebras shows that the bottom row is dense in the top row.
An invariant function $f \in \C[G/U_+ \times G/U_+]^{U_{-}}$ satisfies

\begin{equation}
f(u_1t_1, u_2t_2) = f(t_1, u_1^{-1}u_2t_2),\\
\end{equation}

\noindent
so we can realize $\C[G/U_+ \times G/U_+]^{U_{-}}$ as a subalgebra of $R(G)\otimes \C[T].$ 
The algebra $\C[G/U_+ \times G/U_+]^{U_{-}}$ is multigraded by components which classify the multiplicities of invariants in triple tensor products of irreducible representations of $G.$
On the level of vector spaces we have

\begin{equation}
Hom_G(V(\alpha), V(\beta)\otimes V(\gamma)) = Hom_{U_{-}}(\C v_{\alpha}, V(\beta)\otimes V(\gamma)) = V_{\alpha - \beta ,\beta}(\gamma) \otimes \C v_{\beta}\\
\end{equation}

\noindent
Here $V_{\alpha - \beta ,\beta}(\gamma)$ is the subspace of vectors $v \in V(\gamma)$ such that $h_i(v) = \alpha- \beta(h_i) v$ and $e_i^{\beta(h_i) +1} v = 0$ for $h_i$ the basis member of the Cartan subalgebra of $\mathfrak{g}$ and $e_i$ the raising operator in $\mathfrak{g}$ corresponding to the positive root $\alpha_i.$ For this see chapter XVIII, page
383 of \cite{Zh}.  It is also known that this subspace inherits the dual canonical basis.  The algebra $\C[G/U_+ \times G/U_+]^{U_{-}}$
is isomorphic to the branching algebra $\mathfrak{A}(\Delta_2),$ where $\Delta_2: G \to G\times G$ is the diagonal morphism.  The "lower triangular" multiplication property of dual canonical basis therefore also defines monoidal associated graded algebras of this algebra
in the fashion of \cite{AB}.  To our knowledge these have not been studied.

All three subspaces discussed in this example inherit a basis from the dual canonical basis because it satisfies the so-called good basis property, for this see \cite{Ma}.
We believe that the space of graded valuations of $R(G)$ with respect to the dual canonical basis grading should be a useful combinatorial object for the representation theory of $G.$

\section{Acknowledgements}

We thank Benjamin Giraldo and Dave Anderson for useful conversations, and Dave Anderson
for introducing us to  \cite{LM} and \cite{KKh}.

\date{\today}

\end{document}